\documentclass[a4paper,10pt]{article}
\usepackage[utf8]{inputenc}

\usepackage{comment,xcomment,amssymb,amsmath,color}
\usepackage{rotating,xypic,array}
\usepackage{latexsym,mathtools,amsthm}
\usepackage{subfig}
\usepackage{tikz,fancyvrb}
\usepackage{booktabs}  
\usetikzlibrary{arrows.meta}
\usepackage{enumitem}
\usepackage{url} 
\usepackage{longtable}
\usepackage[section]{placeins}
\usepackage{multirow}
\usepackage{youngtab}
\usepackage{amsmath}
\makeatletter
\def\namedlabel#1#2{\begingroup
\def\@currentlabel{#2}%
\label{#1}\endgroup
}
\makeatother
\title{Linear perturbations of metrics with holonomy $\Spin(7)$}
\author{Diego Conti and Daniel Perolini}

\renewcommand{\Im}{\operatorname{Im}}
\newtheorem{theorem}{Theorem}[section]
\newtheorem{lemma}[theorem]{Lemma}

\newtheorem{proposition}[theorem]{Proposition}
\theoremstyle{definition}

\theoremstyle{remark}
\newtheorem{remark}[theorem]{Remark}

\newcommand{\norm}[1]{\left\Vert#1\right\Vert}

\newcommand{\R}{\mathbb{R}}
\newcommand{\Rnd}{{\mathbb{R}^n}^*}

\newcommand{\im}{\mathrm{Im}\,}         
\newcommand{\lie}[1]{\mathfrak{#1}}     
\newcommand{\g}{\lie{g}}
\newcommand{\Lie}{\mathcal{L}}          

\newcommand{\h}{\mathbb{H}}
\renewcommand{\H}{\mathbb{H}}

\newcommand{\hook}{\lrcorner\,}

\newcommand{\Spin}{\mathrm{Spin}}

\newcommand{\Sp}{\mathrm{Sp}}

\newcommand{\SO}{\mathrm{SO}}

\newcommand{\Gtwo}{\mathrm{G}_2}

\newcommand{\spin}{\mathfrak{spin}}
\newcommand{\GL}{\mathrm{GL}}

\newcommand{\gl}{\lie{gl}}

\newcommand{\Span}[1]{\operatorname{Span}\left\{#1\right\}}

\DeclareMathOperator{\Stab}{Stab}
\DeclareMathOperator{\Ad}{Ad}

\newcolumntype{C}{>{$}c<{$}}
\newcolumntype{L}{>{$}l<{$}}
\newcolumntype{R}{>{$}r<{$}}

\newcommand{\splie}{\lie{sp}}
\newcommand{\contr}{\hook}
\DeclarePairedDelimiter{\norma}{\lVert}{\rVert}
\usepackage{braket}

\begin{document}
\maketitle
\begin{abstract}
We apply the method of linear perturbations to the case of $\Spin(7)$-structures, showing that the only nontrivial perturbations are those determined by a rank one nilpotent matrix. 

We consider linear perturbations of the Bryant-Salamon metric on the spin bundle over $S^4$ that retain invariance under the action of $\Sp(2)$, showing that the metrics obtained in this way are isometric.
\end{abstract}

\renewcommand{\thefootnote}{\fnsymbol{footnote}}
\footnotetext{\emph{MSC class 2020}: \emph{Primary} 53C29; \emph{Secondary} 53C25, 57S15}
\renewcommand{\thefootnote}{\arabic{footnote}}

Riemannian metrics with holonomy $\Spin(7)$ have been studied in differential geometry since the celebrated theorem of Berger \cite{Berger}, listing the possible holonomy groups of an irreducible, nonsymmetric simply connected Riemannian manifold. Metrics with holonomy contained in $\Spin(7)$ are known to be Ricci-flat \cite{Bonan}, and they imply the presence of a parallel spinor \cite{Wang:Parallel}. They are also relevant for string theory (see \cite{GukovSparks}).

The first local examples of metrics with holonomy $\Spin(7)$ were constructed in \cite{Bryant:MetricsWithExceptional}, and the first complete metric was obtained in $\cite{BryantSalamon}$; the latter takes the form of an explicit $\Sp(2)$-invariant metric on the spinor bundle over $S^4$. It was later shown in \cite{Cvetic:NewComplete} that this metric belongs to a one-parameter family of invariant metrics.

We note that the metrics of \cite{BryantSalamon} are of cohomogeneity one;  other cohomogeneity one metrics with holonomy contained in $\Spin(7)$ have been constructed in \cite{CveticGibbonsPope,KannoYasui,GukovSparks,Reidegeld,Cvetic:CohomogeneityOne,Chi,Bazaikin:OnNewExamples,Bazaikin:Noncompact}. Outside of the cohomogeneity one setting other constructions exist, but the metrics they determine are not explicit (see \cite{Joyce:Compact,Joyce:Anew,Foscolo,Kovalev}).

As observed in \cite{Bryant:MetricsWithExceptional}, a metric with holonomy contained in $\Spin(7)$ is defined by a closed form $\Omega$ which is pointwise linearly equivalent to a reference $4$-form on $\R^8$ with stabilizer $\Spin(7)$. It is then possible to define perturbations of a $\Spin(7)$-metric by replacing $\Omega$ with a perturbed form $\Omega+\delta$ which remains pointwise linearly equivalent to $\Omega$. Notice that for the parallel $3$-forms $\varphi$ arising in the context of holonomy $\Gtwo$ the form $\varphi+\delta$ is always linearly equivalent to $\varphi$ for $\delta$ sufficiently small; in other terms, $\varphi$ is stable in the sense of \cite{Hitchin:StableForms}. The  $\Spin(7)$ form $\Omega$ is not stable, however, so more work is needed in order to obtain a perturbation.

One possible approach was considered in \cite[Section 5.2]{Karigiannis} by taking 
\begin{equation} \label{eq:antisimmetric}
	\delta=v^\flat\wedge (w\hook\Omega)-w^\flat\wedge (v\hook\Omega),
\end{equation}
for $v,w$ vector fields on $M$. In terms of the infinitesimal action $\rho$ of $\gl(T_xM)$ on $\Lambda^4 T^*_xM$, this amounts to setting $\delta=\rho(A)\Omega$, where $A$ is the skew-symmetric endomorphism $A=v^\flat\otimes w - w^\flat\otimes v$.  We recall that under $\Spin(7)$ the bundle of four-forms splits as
\begin{equation}
 \label{eqn:split4}
\Lambda^4_1\oplus\Lambda^4_7\oplus\Lambda^4_{27} \oplus \Lambda^4_{35};
\end{equation}
the skew-symmetric $A$ determines a  perturbation term $\delta$ in $\Lambda^4_7$. Whilst this construction gives nontrivial perturbations of the original metric in the case of $\Gtwo$ (mutatis mutandis: the relevant decomposition is $\Lambda^3_1\oplus\Lambda^3_7\oplus\Lambda^3_{27}$
 and the perturbation $\delta$ an element of $\Lambda^3_7$), it turns out that in the $\Spin(7)$ case the perturbed form never defines a $\Spin(7)$-structure (\cite{Karigiannis}).

A different ansatz was considered in \cite{ContiMadsenSalamon} in the context of $\Sp(2)\Sp(1)$-structures, which amounts to imposing that $A$ be nilpotent, rather than skew-symmetric. The key observation, working at a point, is that when 
\begin{equation}
 \label{eq:rhoquadro}
\rho(A)(\rho(A)\Omega)=0,
\end{equation}
the form
\[\Omega + t\delta, \quad \delta= \rho(A)\Omega\]
is always in the same $\GL(8,\R)$-orbit as $\Omega$ for any $t$; one then says that $\delta$ is a linear perturbation of $\Omega$. It turns out (see \cite{ContiMadsenSalamon}) that one can assume $A$ to be nilpotent without loss of generality. 

In this paper we study nilpotent perturbations of the $\Spin(7)$-form $\Omega$. By a case-by-case analysis of the possible Jordan forms of a nilpotent matrix in  $\gl(8,\R)$, and making use of $\Spin(7)$-invariance of \eqref{eq:rhoquadro}, we prove that any linear perturbation of the $\Spin(7)$ form $\Omega$ is defined by a rank one nilpotent matrix, i.e. it has the form \[\delta=v^\flat\wedge (w\hook \Omega),\] with $v,w$ orthogonal  vector fields. In terms of \eqref{eqn:split4}, the resulting perturbations of the $\Spin(7)$ form turn out to be  elements of $\Lambda^4_7\oplus\Lambda^4_{35}$.

We apply the method of linear perturbations to the Bryant-Salamon metric; we construct a family of linear perturbations parameterized by three functions of one variable. However, it turns out that the resulting metrics are isometric; due to the fact that nilpotent perturbations preserve volumes, we do not recover the squashed deformations of \cite{Cvetic:NewComplete}.

Our result complements the result of 
\cite{Lehmann:deformations}, stating that the Bryant-Salamon is rigid  in the class of asymptotically conical $\Spin(7)$ metrics.

\textbf{Acknowledgements.}
This work is partly based on the second author's master thesis \cite{Perolini:thesis}. 
We thank Thomas Madsen for useful discussions.
\section{Linear perturbations}
In this section we classify linear perturbation at a point of $4$-forms defining a $\Spin(7)$-structure, proving that they are in one-to-one correspondence with nilpotent matrices of rank one in $\gl(8,\R)$.

We first recall some results from \cite{ContiMadsenSalamon}.
For a lighter notation, we shall write ${\R^n}^*$ instead of $(\R^n)^*$. Denote by 
\[\gl({\R^n}^*)\times \Lambda^k{\R^n}^*\to\Lambda^k{\R^n}^*, \quad (A,\omega)\mapsto \rho(A)\omega\]
the natural action of $\gl({\R^n}^*)$ on $\Lambda^k{\R^n}^*$. We shall write $\rho(A)^2\omega$ for $\rho(A)(\rho(A)\omega)$. 
\begin{proposition}[\cite{ContiMadsenSalamon}]
Fix $\omega\in\Lambda^k{\R^n}^*$  and a solution $A\in\gl({\R^n}^*)$ of
\begin{equation} \label{eq:deformazioni}
	\rho(A)^2\omega=0.
\end{equation} Then 
\[\beta_t=\omega+ t\rho(A)\omega\]
lies in the same $\GL(n,\R)$-orbit as $\omega$ for all $t\in\R$.
\end{proposition}

It turns out that there is no loss of generality in assuming that $A$ is nilpotent. Indeed, we can apply the Jordan decomposition and write $A=S+N$, where $S$ is semisimple and $N$ is nilpotent. We have the following:
\begin{proposition}[\cite{ContiMadsenSalamon}] \label{prop:nilp}
Let $\omega\in\Lambda^k{\R^n}^*$ and $A\in\gl({\R^n}^*)$ a solution of \eqref{eq:deformazioni} with Jordan decomposition $A=S+N$. Then 
\[\rho(N)\omega=\rho(A)\omega,\; \rho(N)^2\omega=0.\]
\end{proposition}

\begin{remark} \label{claim:1}
	Let $v \in \R^n$, $\alpha \in V^*$ and $\omega \in \Lambda^p {\R^n}^*$. Then 
	$\rho(v \otimes \alpha )\omega=\alpha \wedge (v \hook \omega)$.

	Indeed it suffices to prove the claim for $p=2$: let 
	$\varepsilon_1,\varepsilon_2 \in {\R^n}^*$, then
	\begin{gather*}
	\rho(v \otimes \alpha) \varepsilon_1 \wedge \varepsilon_2 =(v \otimes \alpha) (\varepsilon_1 )\wedge \varepsilon_2+\varepsilon_1 \wedge (v \otimes \alpha) (\varepsilon_2 ) 
	\\= \varepsilon_1(v) \alpha  \wedge \varepsilon_2+\varepsilon_1 \wedge 				\varepsilon_2(v) 	\alpha
	=\alpha \wedge (v \contr \varepsilon_1) \wedge \varepsilon_2
	- \alpha \wedge \varepsilon_1 \wedge (v \lrcorner \, \varepsilon_2)
	\\=\alpha \wedge \bigl(v \lrcorner \, (\varepsilon_1 \wedge \varepsilon_2) \bigr)
	\end{gather*}
	where last equality follows from Leibnitz's rule.
\end{remark}

\begin{remark}\label{claim:2}
	Let $A$ be a nilpotent, rank-one endomorphism
	of $\Rnd$, then
	\begin{equation*}
	\rho(A)^2=0.
	\end{equation*}
	In particular $A$ is a solution of \eqref{eq:deformazioni}
	for all $ \omega$.

	Indeed if $A$ has rank $1$ there exists a basis $v^1, \dots, v^n$
	of $\Rnd$ such that \[Av^1=v^2,\quad  Av^2=\dots =Av^n=0.\] We can write $A$ in tensorial form as $A=v_1 \otimes v^2$, where
	$v_1,\dots,v_n$ is the corresponding dual basis in $\R^n$.
	Let $\omega \in \Lambda^p \Rnd$; we have
	\begin{multline*}
	\rho(v_1 \otimes v^2)^2 \omega=v^2 \wedge \bigg (v_1 \hook \bigl( v^2 
	\wedge (v_1 \hook \omega) \bigr) \bigg)
	\\=v^2 \wedge \bigl ( (v_1 \hook v^2 ) \wedge
	(v_1 \hook \omega)-v^2 \wedge (v_1 \hook v_1 \hook 
	\omega) \bigr)=0
	\end{multline*}
	where the first identity follows from Remark~\ref{claim:1} and the
	second one holds by the Leibnitz rule for $\hook$.
\end{remark}

Recall that if $e_1,\dots, e_8$ is the standard basis of
$\R^8$ and $\alpha,\beta, \Omega$ are the linear forms
defined by
\begin{align*}
	\alpha &=e^{12}+e^{34}+e^{56}+e^{78},\\
	\beta &=(e^1+ie^2) \wedge (e^3+ie^4) \wedge (e^5+ie^6) \wedge (e^7+ie^8),\\
	\Omega &=\frac{\alpha^2}{2}+Re(\beta),
\end{align*}
then the stabilizer in $\GL(8,\R)$ of the $4$-form $\Omega$ 
is a subgroup of $\SO(8)$ isomorphic to $\Spin(7)$ (see \cite{Bonan,Bryant:MetricsWithExceptional}). Moreover, $\Spin(7)$ acts  transitively on the sphere $S^7 \subset \R^8$, and the stabilizer of $e_8$ is isomorphic to $G_2$, which acts transitively on the sphere $S^6 \subset \R^7\cong \R^7 \times \{0\}$. From now on we shall make the identifications $\Spin(7)=\Stab(\Omega)$,  $G_2=\Stab(\Omega)\cap\Stab (e_8)$. Giving 
a $\Spin(7)$-structure on a $8$-manifold amounts to giving a $4$-form linearly equivalent to $\Omega$ at each point.

Thus, we are interested in linear perturbations of $\Omega$; in particular, we set $n=8$ and $k=4$.
Up to change of basis, nilpotent matrices are classified over the reals
by partitions with weight $8$, giving $22$ possibilities that can be encoded in terms of
Young diagrams. For example, the diagram
\[
\Gamma=\yng(3,2,1,1,1) \:.
\]
describes an endomorphism of ${\R^8}^*$
with Jordan blocks of size $(3, 2, 1, 1, 1)$, which,
with respect to some basis $\{w^1,v^2,v^3,w^4,v^5,v^6,v^7,v^8\}$, satisfies
\begin{gather*}
	w^1\mapsto v^2 \mapsto v^3 \mapsto 0\\
	w^4 \mapsto v^5 \mapsto 0\\
	v^6,v^7,v^8 \mapsto 0.
\end{gather*}
In the rest of this paper we will use the notation
illustrated in the last example:
for each Jordan block $J_i$ of dimension $r \ge 2$ we fix an element $w^i$
such that $w^i, Aw^i,\dots, A^{r-1}w^i$ are linearly indipendent, and denote the other basis elements by $v^j$. The dual basis of $\R^n$ will be denoted by $\{w_i,v_j\}$.

In the following, we will need to consider the Young diagrams
\begin{equation*}
\Gamma_1=\yng(3,2,2,1) \ \  ,  \ \
\Gamma_2=\yng(2,2,2,2) \ \  ,    \ \
\Gamma_3=\yng(2,2,2,1,1) \ \  ,  \ \
\Gamma_4=\yng(2,2,1,1,1,1) \ \  ,  \ \
\Gamma_5=\yng(2,1,1,1,1,1,1) \ \  ,  \ \
\Gamma_6=\yng(1,1,1,1,1,1,1,1) \  .
\end{equation*}
describing six particular configurations
of Jordan blocks. Notice that $\Gamma_5$ corresponds to rank-one nilpotent endomorphisms and $\Gamma_6$ to zero.

Given a four-form $\omega$ on $\R^8$ and two vectors $u,v\in\R^8$, we will say that the contraction $u\hook v\hook \omega$ is degenerate if so is the bilinear form induced on the quotient $\R^8/\Span{u,v}$, i.e. 
\[(u\hook v \hook\omega)^3\neq0.\]
\begin{lemma} \label{lemma:degeneracy}
Fix $\omega \in \Lambda^4{\R^8}^*$ and let $A\in\gl({\R^n}^*)$ be a nilpotent solution of \eqref{eq:deformazioni}. If
	$A$ has diagram $\Gamma_1, \Gamma_2,\Gamma_3, \Gamma_4$ and $\{w^i,v^k\}$
	is a Jordan basis of $A$, then
	\begin{equation}
		w_i \hook w_j \hook \omega \text{ is degenerate for all } i,j.
	\end{equation}
\end{lemma}
\begin{proof} 
	\emph{Case $\Gamma_2$}:
	writing $A$ in tensorial form we have
		\begin{equation*}
			A=\sum_{i=1}^{4} w_i \otimes v^i.
		\end{equation*}
		The following hold:
		\begin{gather*}
			\rho(A)^2 \omega=\rho \bigg(\sum_{i=1}^{4} w_i \otimes v^i \bigg)^2 \omega=
			\sum_{1\leq i,j \leq 4}\rho(w_i \otimes v^i)\rho(w_j \otimes v^j)\omega
			\\=2 \sum_{1\leq i<j \leq 4} v^i \wedge v^j \wedge (w_i \hook w_j \hook \omega).
		\end{gather*}
		The second equality follows
		from the identities
		\begin{align*}
			\rho(w_i \otimes v^i)\rho(w_j \otimes v^j)&=\rho(w_j \otimes v^j)\rho(w_i \otimes v^i),\\
			\rho(w_i \otimes v^i)^2&=0,
		\end{align*}
		(easy consequences of Remark~\ref{claim:1} and Remark~\ref{claim:2}), and the last equality
		holds because of Remark~\ref{claim:1}.	Thus, we can  write		\eqref{eq:deformazioni} in the form
		\begin{equation}\label{eq:deformation}
			\sum_{1\leq i<j \leq 4} v^{ij} \wedge (w_i \hook w_j \hook
			\omega)=0.
		\end{equation}
		Contracting by $w_k$, multiplying with $v^l$ and using
		Remark~\ref{claim:1} and Remark~\ref{claim:2}
		we obtain the following identities:
		\begin{equation} \label{eq:ijkl}
			 v^{lij} \wedge (w_k \hook w_i \hook w_j \hook
			\omega)=0 \quad \quad \forall i,j,k,l \quad:\quad \{i,j,l,k \}=\{1,2,3,4\}.
		\end{equation}
		We can decompose $\omega$ as
		\begin{gather}	\label{eq:omegariscritta}
				\omega=\sum_{i=1}^{4} w^i \wedge \alpha_i + \sum_{1 \leq i < j\leq 4}
				w^{ij} \wedge \beta_{ij}+\sum_{1 \leq i < j<k \leq 4} 	w^{ijk} \wedge \gamma_{ijk}
			+\delta w^{1234} +\varepsilon
		\end{gather}
			\begin{gather*}
		\alpha_i,\beta_{ij},\gamma_{ijk},\varepsilon \in \Lambda \;\Span {v^1, \dots, v^4} \quad,\quad \delta \in \R	.
		\end{gather*} 
		We have that \eqref{eq:ijkl} implies
		\begin{equation}\label{eq:deltagamma}
			\delta=0  ,\quad \gamma_{ijk}=c^lv^l \quad\quad \forall\; \{i,j,l,k \}=\{1,2,3,4\},\quad c^l \in \R. 
		\end{equation} 
		Notice that in order to prove the degeneracy
		of $w_i \hook w_j \hook \omega$ 
		it is sufficient to prove $c^l=0$ for $ l=1,2,3,4$.
		Substituting \eqref{eq:omegariscritta} and \eqref{eq:deltagamma} in \eqref{eq:deformation}
		and writing \eqref{eq:deformation}
		in the form
		$w^1 \wedge I_1+w^2 \wedge I_2+w^3 \wedge I_3+
		w^4 \wedge I_4=0$ it turns out that
		$I_1=0=I_2=I_3=I_4$; this implies
		the linear system		\begin{equation*}
		\begin{pmatrix}
		1 & -1 & 0 & -1 \\ 
		0 & -1 & 1 & -1 \\ 
		1 & 0 & -1 & 1 \\ 
		-1 & 1 & -1 & 0
		\end{pmatrix} 
		\begin{pmatrix}
		c^1 \\ 
		c^2 \\ 
		c^3 \\ 
		c^4
		\end{pmatrix}=0;
		\end{equation*} 
		by  nonsingularity of the matrix, we have $c^l=0$,
		$l=1,\dots,4$.
		
		\emph{Case $\Gamma_1$}:
		this time we have
		\begin{equation*}
			A=v_4 \otimes v^1+\sum_{i=1}^{3} w_i \otimes v^i.
		\end{equation*}
		Arguing as in case $\Gamma_2$, Equation~\eqref{eq:deformazioni} can be written as
		\begin{equation}	\label{eq:casogamma1}
			2\sum_{1 \leq i < j \leq 3}v^{ij} \wedge (w_i \hook w_j \hook \omega)
			+v^4 \wedge \bigg( w_1 \hook \omega +
			2 \sum_{i=1}^{3} v^i \wedge (w_i \hook v_1 \hook \omega) \bigg)=0.
		\end{equation}
		Multiplying by $v^4$ and contracting by $w_k$ with
		$k=1,2,3$ gives
		\begin{equation} \label{eq:defcontr}
			v^{ij4} \wedge (w_3 \hook w_2 \hook w_1 \hook \omega)=0 \quad  \forall \; 1 \le i <j \le 3.
		\end{equation}
		Similarly as in the case of $\Gamma_2$, we write
		\begin{gather}	\label{eq:omegariscritta2}
		\omega=\sum_{i=1}^{3} w^i \wedge \alpha_i + \sum_{1 \leq i < j\leq 3}
		w^{ij} \wedge \beta_{ij}+ w^{123} \wedge \gamma_{123}
		+\delta w^{123} + \varepsilon,
		\end{gather}
		and \eqref{eq:defcontr} gives 
		\begin{equation} \label{eq:lambda}
				\gamma_{123}=\lambda v^4 \quad, \quad \lambda \in \R.		
		\end{equation}
		It is sufficient to prove $\lambda=0$:
		substituting \eqref{eq:omegariscritta2},
		\eqref{eq:lambda} in 
		\eqref{eq:defcontr} and writing 
		\[
			\beta_{12}=\sum_{i<j} y_{ij}v^{ij};\quad
			 \beta_{13} =\sum_{i<j}x_{ij}v^{ij};\quad
\beta_{23}=\sum_{i<j} z_{ij}v^{ij},\]
we obtain an equation of the form 
		\[
				I_1v^{124} \wedge w^3+I_2v^{134} \wedge w^2+
		I_3v^{234} \wedge w^1+\dots=0 \quad \quad I_1,I_2,I_3\in \R,
		\]
		resulting in $I_1=0=I_2=I_3$;
		explicitly, we have the linear system
		\begin{equation*}
		\begin{pmatrix}
		1 & 3 & 0 \\ 
		-1 & 0 & 3 \\ 
		1 & 2 & -2
		\end{pmatrix}
		\begin{pmatrix}
		\lambda \\ 
		x_{12} \\ 
		y_{13}
		\end{pmatrix} =0,
		\end{equation*}
		with nonsingular matrix, so $\lambda=x_{12}=y_{13}=0$.
		
		Cases $\Gamma_3, \Gamma_4$ are similar (and easier).
	\end{proof}

We will need the following:
\begin{proposition}
\label{prop:nondeg}
	Let $u,v \in \R^8$ be linearly indipendent. Then $u \hook v \hook \Omega$ is nondegenerate.
\end{proposition}
\begin{proof}
It is sufficient to prove the thesis with
	$u,v$ orthogonal and normalized, because the following hold:
	\begin{align*}
		(u \hook v \hook \Omega)^3&=\norma{u}^3\norma{v}^3
		\bigg(\frac{u}{\norma{u}}  \hook  
		\frac{v}{\norma{v}} \hook \Omega\bigg)^3,\\
 		u \hook v \hook \Omega&=u \hook \big(v-P_{u}v \big) \hook \Omega,
	\end{align*}
	where $P_u$ is the orthogonal projection onto the subspace
	generated by $u$.
	So let $u, v$ be orthogonal vectors in $S^7$; since $\Spin(7)$ acts transitively on $S^7$, there exists
	$R_1 \in\Spin(7)$ such that $R_1v=e_8$;
	in particular $R_1$ is an isometry, so
	$R_1u \perp R_1v=e_8$ and $R_1u \in \R^7$.
	It follows that $R_1u \in S^6$, but $G_2$
	is transitive on $S^6$ so
	there exists $R_2 \in G_2$ such that
	$R_2R_1u=e_7$.
	Setting $R=R_1^{-1}R_2^{-1}$ we have
	$u=Re_7$ and $v=Re_8$.
	 For all $x,y \in \R^8$ we have 
\begin{multline}	
	(u\hook v \hook \Omega)(x,y)
	=\Omega(Re_7,Re_8,x,y)
	=\Omega(e_7,e_8,R^{-1}x,R^{-1}y) \\
	 =(e_7\hook e_8 \hook \Omega)(R^{-1}x,R^{-1}y)
	=(R^{-1})^*(e_7\hook e_8 \hook \Omega)(x,y);
	\label{eq:contrazioni}
\end{multline}
	the second equality holds
	by the $\Spin(7)$-invariance of $\Omega$.
	So from \eqref{eq:contrazioni} 
	we have
	\begin{equation*}
		(u\hook v \hook \Omega)^3=(R^{-1})^*(e_7\hook e_8 \hook \Omega)^3,\end{equation*}
	but
	\begin{gather*}
	(e_7\lrcorner e_8 \lrcorner \Omega)^3=(e^{35}+e^{48}+e^{67})^3
	=6e^{354867} \neq 0.\qedhere
	\end{gather*}
\end{proof}
We can finally prove:
\begin{theorem} \label{teo:linpert}
	If $\rho(A)\Omega$ is a linear perturbation of $\Omega$, i.e. $\rho(A)^2\Omega=0$, then the nilpotent part of $A$ has rank at most one.
\end{theorem}
\begin{proof}
For each diagram $\Gamma$, we can fix a representative
endomorphism $A_{\Gamma}$ and compute the space
\begin{equation*}
	K_{\Gamma}=\Set{\omega \in \Lambda^4 {\R^8}^*| \rho(A_{\Gamma})^2\omega=0}.
\end{equation*}
The equation $\rho(A)^2\Omega=0$ has a solution with diagram 
$\Gamma$ if $\rho(A_{\Gamma})^2\omega=0$ for some $\omega$
in the same $\GL(8,\R)$-orbit as $\Omega$; by Proposition~\ref{prop:nondeg}, this implies
that for any linearly indipendent vectors $u,v\in\R^8$
the map 
\begin{gather*}
	K_{\Gamma} \rightarrow \Lambda^4 {\R^8}^* \\
	\omega \mapsto (u \hook v \hook \omega)^3
\end{gather*}
is not identically zero. As observed in \cite{ContiMadsenSalamon}, this rules out all cases except
$\Gamma_1,\dots,\Gamma_6$. Let $A$ be a solution with
$\Gamma$ one of the remaining diagrams. Using again Proposition~\ref{prop:nondeg},
	we have that $u \hook v \hook \Omega $ is nondegenerate for any choice of linearly independent vectors $u,v\in\R^8$; it follows from Lemma~\ref{lemma:degeneracy}
	that all nilpotent solutions
	of $\rho(A)^2\Omega=0$ are either zero or rank-one nilpotent
	endomorphisms. 
\end{proof}

\begin{remark}
Linear perturbations of a $\Spin(7)$ form lie in the module $\Lambda^4_7\oplus\Lambda^4_{35}$.
Indeed, the map 
\[\lie{sl}(8,\R)\to \Lambda^4\R^8, \quad A\mapsto \rho(A)\Omega,\]
is $\Spin(7)$-equivariant and its kernel $\spin(7)$ has dimension $21$; the image is therefore the only $\Spin(7)$-module of dimension $42$ inside $\Lambda^4\R^8$.

Notice that we cosider $\lie{sl}(8,\R)$ instead of $\gl(8,\R)$ because we assume $A$ to be nilpotent.
\end{remark}

\section{A cohomogeneity one description of the Bryant-Salamon metric}
Recall from \cite{BryantSalamon} that the spinor bundle $S$ over $S^4$ carries a cohomogeneity one metric with holonomy $\Spin(7)$; this metric has cohomogeneity one under the action of $\Sp(2)$. In this section we give a description of these metrics in terms of cohomogeneity one actions which will be needed in order to study the linear perturbations.

Explicitly, the Lie group $\Sp(2)=\{g\in \GL(2,\h)\mid gg^*=I\}$ contains two copies of $\Sp(1)$, i.e.
\[
\Sp(1)_+=\left\{\begin{pmatrix} p & 0 \\ 0 & 1 \end{pmatrix} | p\in\Sp(1)\right\}, \quad 
\Sp(1)_-=\left\{\begin{pmatrix} 1 & 0 \\ 0 & q \end{pmatrix} | q\in\Sp(1)\right\}.\]
At the Lie algebra level, 
\[
            \mathfrak{sp}(2)=\Set{\begin{pmatrix}
            a & b \\
            -\overline{b} &c 
            \end{pmatrix}}, 
            \lie{sp}_+=\Set{\begin{pmatrix}
            a & 0 \\
            0 & 0\end{pmatrix} }, 
            \lie{sp}_-=\Set{\begin{pmatrix}
            0 & 0 \\
            0 & c\end{pmatrix} },
\]
with $a,c\in\Im\h$, $b\in\h$.

The spinor bundle $S$ has the form
\[S=(\Sp(2)\times\h) / (\Sp(1)_+\times \Sp(1)_-),\]
where $(p,q)\in\Sp(1)_+\times \Sp(1)_-$ acts on the right by
\[(g,v)(p,q)=(g(p,q),p^{-1}vq).\]
$S$ is of cohomogeneity one under the action of $\Sp(2)$; there is one singular orbit, namely $\Sp(2)/\Sp(1)_+\times\Sp(1)_-=S^4$, and  the complement of the singular orbit has the form
\[S\setminus S^4=\Sp(2)/\Sp(1)_+\times \R_+.\]
Notice that the following 
    \begin{equation}
    \label{eq:basis} 
    \begin{gathered}	
    A_1=\frac{1}{\sqrt{12}}\begin{pmatrix}
    i & 0 \\
    0 &0 
    \end{pmatrix},
    A_2=\frac{1}{\sqrt{12}}\begin{pmatrix}
    j & 0 \\
    0 &0 
    \end{pmatrix},
    A_3=\frac{1}{\sqrt{12}}\begin{pmatrix}
    k & 0 \\
    0 &0 
    \end{pmatrix},\\	
    A_4=\frac{1}{\sqrt{12}}\begin{pmatrix}
    0 & 0 \\
    0 &i 
    \end{pmatrix},
    A_5=\frac{1}{\sqrt{12}}\begin{pmatrix}
    0 & 0 \\
    0 &j 
    \end{pmatrix},
    A_6=\frac{1}{\sqrt{12}}\begin{pmatrix}
    0 & 0 \\
    0 &k 
    \end{pmatrix},\\	
    X_1=\frac{1}{\sqrt{24}}\begin{pmatrix}
    0 & i \\
    i &0 
    \end{pmatrix},
    X_2=\frac{1}{\sqrt{24}}\begin{pmatrix}
    0 & j \\
    j &0 
    \end{pmatrix},
    X_3=\frac{1}{\sqrt{24}}\begin{pmatrix}
    0 & k \\
    k &0 
    \end{pmatrix},
        X_4=\frac{1}{\sqrt{24}}\begin{pmatrix}
        0 & 1 \\
        -1 &0 
        \end{pmatrix}.
    \end{gathered}
\end{equation}
is an orthonormal basis of $\lie{sp}(2)$ with respect to the Killing metric.
Let $a=a_0 +ia_1+ja_2+ka_3$ be the standard real coordinates
in $\H$; following \cite{BryantSalamon}, we define $\h$-valued one-forms on $\Sp(2)\times\h$
    \begin{equation*}
        \phi=iA^4+jA^5+kA^6, \quad \omega=X^4+iX^1+jX^2+kX^3,\quad 
        \alpha=d a-a\phi;
    \end{equation*}
we then define $\im\h$-valued two-forms 
\[B=\frac{1}{2} (\bar{\alpha} \wedge \alpha), \quad 
 \Omega=\frac{1}{2} (\bar{\omega} \wedge \omega).
\]
When needed, we will use  indices to indicate components in $\H$, i.e.
    \begin{gather*}
iB_1+jB_2+kB_3=i(\alpha_0 \wedge \alpha_1-\alpha_2 \wedge \alpha_3)+j(\alpha_0 \wedge \alpha_2-\alpha_3 \wedge \alpha_1)+k(\alpha_0 \wedge \alpha_3-\alpha_1 \wedge \alpha_2).        
    \end{gather*}
The Bryant-Salamon $4$-form is a linear combination of  the forms
    \begin{gather*}
        \psi_1=\alpha_0 \wedge \alpha_1 \wedge \alpha_2 \wedge\alpha_3,\quad
        \psi_=B_1 \wedge \Omega_1+ B_2 \wedge \Omega_2+
        B_3 \wedge \Omega_3,\\
        \psi_3=\omega_0 \wedge \omega_1 \wedge \omega_2 \wedge
        \omega_3,
    \end{gather*}
with coefficients determined by the smooth functions on $ \H$
    \begin{equation*}
        f(r)=4\bigl(1+r\;\bigr)^{-2/5}, \quad g(
        r)=5k\bigl(1+r \;\bigr)^{3/5},
    \end{equation*}
where we have set $r=a\bar{a}=\norm a ^2$. More precisely, the Bryant-Salamon $4$-form $\Phi \in \Omega^4\bigl(        \Sp(2) \times \H \bigr)$ is defined as        \begin{equation}            \Phi=f^2\psi_1+fg\psi_2+g^2\psi_3.        \end{equation}
Since $\Phi$ is basic relative to the action of $\Sp(1) \times \Sp(1)$, it induces a form on the quotient $S=\Sp(2) \times \H/\Sp(1) \times \Sp(1)$, also to be denoted by $\Phi$.
    
	\begin{proposition} \label{prop:BSform1}
Under the inclusion
\[
			\tilde{\chi}\colon \Sp(2) \times \R_+ \rightarrow \Sp(2) \times \H, \quad 
			(g,t) \mapsto (g,\sqrt{t}),\]
the Bryant-Salamon $4$-form pulls back to 
		\begin{gather*}
		 \tilde{\chi}^*\Phi=-\frac{d t}{2} \wedge \bigg(tf(t)^2A^{456}+
		f(t)g(t)\bigl(A^4 \wedge (-X^{14}-X^{23})+
		A^5 \wedge (-X^{24}+X^{13})\\+
		A^6 \wedge (-X^{34}-X^{12})\bigr) \bigg)
		-tf(t)g(t)\bigl(A^{56} \wedge (-X^{14}-X^{23})+
		A^{64} \wedge (-X^{24}+X^{13})\\+
		A^{45} \wedge (-X^{34}-X^{12})\bigr)-g(t)^2X^{1234}.
		\end{gather*}
	\end{proposition}
\begin{proof}
By definition we have 
	$\tilde{\chi}^*a=\sqrt{t}$, $\tilde{\chi}^*r=t$, so
	\begin{equation*}
		\tilde{\chi}^*\alpha= \frac{d t}{2\sqrt{t}}-i\sqrt{t}\phi_1-j\sqrt{t}\phi_2-k\sqrt{t}\phi_3.
	\end{equation*}
We obtain
	\begin{align*}
		\tilde{\chi}^*\psi_1&=-t\frac{d t}{2} \wedge \phi_1 \wedge \phi_2 \wedge \phi_3=-t\frac{d t}{2} \wedge A^{456};\\	
		\tilde{\chi}^*\psi_2
&=-\frac{d t}{2} \wedge
		(\phi_1 \wedge \Omega_1+\phi_2 \wedge \Omega_2+\phi_3\wedge \Omega_3)
		-t(\phi_2 \wedge \phi_3 \wedge \Omega_1+\phi_3 \wedge \phi_1 \wedge \Omega_2+\phi_1 \wedge \phi_2 \wedge \Omega_3)\\	 
&=-\frac{d t}{2}\wedge\bigl( 
A^4  \wedge(-X^{14}-X^{23})+
	  A^5 \wedge (-X^{24}+X^{13})+
	  A^6 \wedge (-X^{34}-X^{12})\bigr)	\\
&\quad -t\left(A^{56} \wedge (-X^{14}-X^{23})+
	  A^{64} \wedge (-X^{24}+X^{13})+
	  A^{45} \wedge (-X^{34}-X^{12})\right);\\
\tilde{\chi}^*\psi_3&=	\omega_0 \wedge \omega_1 \wedge \omega_2 \wedge	\omega_3=-X^{1234}.
	\end{align*}
The statement follows immediately.	
\end{proof}

\section{Linear perturbations of the Bryant-Salamon metric}
In this section we study $\Sp(2)$-invariant linear perturbations of the Bryant-Salamon metric.

By Theorem~\ref{teo:linpert}, a linear perturbation of a $\Spin(7)$-structure is obtained by the choice of a rank one nilpotent endomorphism of the tangent bundle at each point. Thus, the global data for a linear perturbation is given by the choice of a vector field $X$ and a one form $\alpha$ with $\alpha(X)=0$. Since we work in the $\Sp(2)$-invariant setting, we will require both vector field and form to be invariant.

Thus, the first step is to construct an $\Sp(2)$-invariant vector field on the cohomogeneity one manifold $\Sp(2)/\Sp(1)_+\times\R_+$. We will need the following observation:
		\begin{lemma} Let a Lie group $G$ act transitively on $M$, and let $H$ be the stabilizer at a point $m$. Then $X \in \mathfrak{g}$ defines a $G$-invariant vector field on $M$ of the form 
			\begin{equation*}
				X^+_{gm}=g^* \frac{d}{dt}|_{t=0} \; \exp(tX)m
			\end{equation*} 
			  if and only if $X$ belongs to 
			\begin{equation} \label{eq:suitable}
				\mathfrak{n}(H)=\Set{X \in \g |\Ad_hX-X \in \mathfrak{h} \quad \forall h \in H}
			\end{equation}
All $G$-invariant vector fields on $M$ are of this form.
		\end{lemma}
		\begin{proof}
The vector field $X^+$ is well defined and invariant if and only if $X^+_{gm}=X^+_{g'm}$ whenever
 $gm=g'm$; in other words, we need $X^+_{gm}=X^+_{ghm}$ for all $g \in G, h \in H$. Since
\begin{gather*}
X^+_{ghm}=g^*h^*\frac{d}{dt}|_{t=0} \; \exp(tX)m =	g^*\frac{d}{dt}|_{t=0} \; h \;exp(tX) h^{-1}m\\				=g^*\frac{d}{dt}|_{t=0} \; \exp(t\Ad_hX)m,
\end{gather*}
we obtain that $X^+$ is well defined when
			\begin{equation*}
				\frac{d}{dt}|_{t=0} \; exp(tX)m =\frac{d}{dt}|_{t=0} \; \exp(t\Ad_hX)m \quad \forall h \in H,
			\end{equation*}	
which is equivalent to $X$ lying in $\lie{n}(H)$.

Conversely, given an invariant vector field $Y$ on $M$, we have $Y_{gm}=g^*Y_m$, where 
\[Y_m=\frac{d}{dt}|_{t=0} \exp(tX)m, \quad X\in\lie g.\]
By invariance, $Y=X^+$, and $X$ lies in $\lie{n}(H)$ by the first part.
\end{proof}
	
\begin{proposition}\label{prop:normalizer}
Relative to the action of $G=Sp(2)$ on $Sp(2)/\Sp(1)_+$ we have
		\begin{equation}
			\mathfrak{n}(\Sp(1)_+)=\splie(1)_+ \times \splie(1)_-.
		\end{equation}
	\end{proposition}
\begin{proof}
An element
	\begin{equation*}
	\begin{pmatrix}
	x & y \\
	-\bar{y} &w 
	\end{pmatrix}\in\splie(2)
	\end{equation*}
lies in $\lie{n}(\Sp(1)_+)$ if and only if for all $p$ in $Sp(1)$ we have 
	\begin{equation*}
	\begin{pmatrix}
	p & 0 \\
	0 &1 
	\end{pmatrix}
	\begin{pmatrix}
	x & y \\
	-\bar{y} &w 
	\end{pmatrix}
	\begin{pmatrix}
	\bar{p} & 0 \\
	0 &1 
	\end{pmatrix}-
	\begin{pmatrix}
	x & y \\
	-\bar{y} &w 
	\end{pmatrix}\in \splie(1)_+.
	\end{equation*} 
This is equivalent to $py=y$ for all $p$, i.e. $y=0$, so the statement is proved.
\end{proof}
Summing up, we have a linear map
\[\splie(1)_+\times\splie(1)_-\to \mathfrak{X}_{Sp(2)}(Sp(2)/Sp(1)_+), \quad X\mapsto X^+;\]
its kernel is $\splie(1)_+$, showing that invariant
vector fields on $\Sp(2)/\Sp(1)_+$ take the form 
 \[X=x_4A_4+x_5A_5+x_6A_6 \in \splie(1)_-.\]
\begin{lemma}
\label{lemma:invariantfield}
Every $\Sp(2)$-invariant vector field $Y$ on $\Sp(2)/\Sp(1)_+$ satisfies
\begin{align*}
\Lie_Y&\left(A^{56} \wedge (-X^{14}-X^{23})+
	  A^{64} \wedge (-X^{24}+X^{13})+
	  A^{45} \wedge (-X^{34}-X^{12})\right)= 0\\  \Lie_Y& X^{1234}=0	  \end{align*}
\end{lemma}
\begin{proof}
As an $\Sp(1)_-$-module, $\splie(2)$ decomposes as
\[3\R+ \h+\splie(1)_-=\Span{A_1,A_2,A_3}+ \Span{X_1,X_2,X_3,X_4}+\Span{A_4,A_5,A_6}
,\]
with $\Lambda^2\h$ splitting as $3\R+\splie(1)_-$. The inclusion of $\splie(1)_-$ in $\Lambda^2\h$ is realized by the $\Sp(1)_-$-equivariant map 
\[A_4\mapsto \Omega_1, A_5\mapsto \Omega_2, A_6\mapsto \Omega_3.\]
It follows that $A^{56}\wedge\Omega_1+A^{64}\wedge\Omega_2 + A^{45}\wedge\Omega_3$ is $\Sp(1)_-$-invariant. As an element of $\Lambda^4\H$, $X^{1234}$ is also $\Sp(1)_-$-invariant. 

Writing $Y=aA_4^++bA_5^++cA_6^+$, the statement follows.
\end{proof}

\begin{theorem}
Given smooth even functions $a,b,c\colon\R\to\R$, the $4$-form
\begin{equation}
\label{eqn:perturbed}
\Phi + dt\wedge (a(t)A_4^++b(t)A_5^++c(t)A_6^+)\hook\Phi
\end{equation}
is closed and defines an $\Sp(2)$-invariant metric with holonomy contained in $\Spin(7)$.
\end{theorem}
\begin{proof}
Observe first that the vector field $a(t)A_4^++b(t)A_5^++c(t)A_6^+$ is globally defined and vanishes on the special orbit, i.e. the zero section of the spinor bundle.

The $4$-form \eqref{eqn:perturbed} is a linear perturbation of the Bryant-Salamon form by a nilpotent endomorphism of rank one (see Remark~\ref{claim:2}), so it defines again a $\Spin(7)$-structure. In order to check that it is closed, write
$Y=a(t)A_4^++b(t)A_5^++c(t)A_6^+$; we have
\[d(dt\wedge Y\hook\Phi)=-dt\wedge d (Y\hook\Phi)=-dt\wedge\Lie_Y\Phi.\]
Since $Y$ is $\Sp(2)$-invariant, by Lemma~\ref{lemma:invariantfield} we have that $\Lie_Y$ annihilates the restriction of $\Phi$ to each principal orbit $\{t=t_0\}$, and therefore $dt\wedge\Lie_Y\Phi=0$. By \cite{Fernandez:AClassification}, the metric defined by the perturbed form has holonomy contained in $\Spin(7)$.
\end{proof}

It is now natural to ask whether the perturbed metrics are isometric to the Bryant-Salamon metric. It turns out that they are isometric under an $\Sp(2)$-equivariant diffeomorphism, due to the following:
\begin{lemma}
Any $\Sp(2)$-invariant vector field on $S$ is a Killing field for the Bryant-Salamon metric.
\end{lemma}
\begin{proof}
The Bryant-Salamon metric takes the form
\begin{multline*}
f(\alpha_0^2+\dots + \alpha_3^2)+g(\omega_0^2+\dots + \omega_3^2)\\
=f(\frac1{4t}dt^2+t((A^4)^2+(A^5)^2+(A^6)^2))+g((X^1)^2+(X^2)^2+(X^3)^2+(X^4)^2)
.
\end{multline*}
Since $(A^4)^2+(A^5)^2+(A^6)^2$ and $(X^1)^2+(X^2)^2+(X^3)^2+(X^4)^2$ are $\Sp(1)_-$-invariant, the claim follows.
\end{proof}
Arguing as in \cite[Proposition 5.2]{ContiMadsenSalamon}, we obtain:
\begin{proposition}
The $\Sp(2)$-invariant linear perturbations of the Bryant-Salamon metric are obtained from the Bryant-Salamon metric via an $\Sp(2)$-equivariant diffeomorphism.
\end{proposition}

\small\noindent D.~Conti: Dipartimento di Matematica e Applicazioni, Universit\`a di Milano Bicocca, via Cozzi 55, 20125 Milano, Italy.\\
\texttt{diego.conti@unimib.it}\\
D.~Perolini: Dipartimento di Matematica F.~Casorati, Università di Pavia, via Ferrata 5, 27100 Pavia, Italy.\\
\texttt{daniel.perolini2@gmail.com}\\

\bibliographystyle{plain}

\bibliography{linear}

\end{document}